\documentclass{article}
\usepackage[utf8]{inputenc}
\usepackage[english]{babel}
\usepackage{amsmath}
\usepackage{amssymb}
\usepackage{amsthm}
\usepackage{enumerate}
\usepackage{color,graphicx}
\usepackage{lastpage209}

\numberwithin{equation}{section}

\theoremstyle{thmstyleone}%
\newtheorem{theorem}{Theorem}[section] 
\newtheorem{lemma}[theorem]{Lemma}
\newtheorem{corollary}[theorem]{Corollary}
\newtheorem{proposition}[theorem]{Proposition}
\theoremstyle{definition}
\newtheorem{definition}[theorem]{Definition}
\newtheorem{remark}[theorem]{Remark}

\newtheorem{example}[theorem]{Example}

\numberwithin{equation}{section}

\def\F{{\mathbb F}}

\def\A{\mathcal A}

\def\L{\mathcal L}

\def\SS{\mathcal S}

\def\char{{\rm char}\, }

\def\dim{{\rm dim}\,}

\def\R{{\mathbb R}}

\title{Steady growth of length function and Malcev algebras \footnote{The work was financially supported by the grant RSF 21-11-00283.}}

\author{A. Guterman, D. Kudryavtsev}
\date{}

\newcommand{\Addresses}{
	\bigskip
	\footnotesize
	
	\textsc{Alexander Guterman, }{Faculty of Algebra, Department of
		Mechanics and Mathematics, Lomonosov Moscow State University, Moscow
		119991, Russia; Moscow Center for Fundamental and Applied Mathematics,  Moscow,   119991, Russia; Moscow Center for Continuous Mathematical Education,  Moscow, 119002, Russia
	}\par\nopagebreak
	\textit{E-mail address}:   \texttt{guterman@list.ru}
	
	\medskip

	\textsc{Dmitry Kudryavtsev, }{School of Mathematics, University of Manchester, Manchester M13 9PL, UK; Moscow Center for Fundamental and Applied Mathematics,  Moscow, 119002, Russia}\par\nopagebreak
	\textit{E-mail address}:   \texttt{dmitry.kudryavtsev@postgrad.manchester.ac.uk}
}

\begin{document}
\maketitle

\Addresses

\begin{abstract}We introduce and investigate the algebras of steadily growing length, that is the class of algebras, where the length is bounded by a linear function of the dimension. In particular we show that Malcev algebras belong to this class and establish the exact upper bound for its length.

\end{abstract}


\section{Introduction}

Let $\F$ be an arbitrary field. In this paper $\A$ denotes a finite dimensional not necessarily unital not necessarily associative $\F$-algebra with the multiplication $(\cdot)$ usually denoted by the concatenation. Let $\SS=\{a_1,\ldots,a_k\}$ be a finite generating set of $\A$. Any product of a finite number of elements from $\SS$ is a {\em word} in $\SS$. The {\em length} of the word $w$, denoted $l(w)$, equals to the  number of letters in the corresponding product.
It is worth noting that different choices of brackets provide different words of the same length due to the non-associativity of $\A$.  If $\A$ is unital, we consider $1$ as a word in $\SS$ with the {\em length $0$}.

The set of all words  in $\SS$ with the lengths less than or equal to $i$ is denoted by $\SS^i$, here $i\ge 0$.

Note that similarly to the associative case, $m<n$ implies that $\SS^m \subseteq \SS^n$.

The set $\L_i(\SS) = \langle \SS^i \rangle$  is the linear span  of  the set  $\SS^i$
(the set of all finite linear combinations with coefficients belonging to
$\mathbb{F}$). We write $\L_i$ instead of $\L_i(\SS)$ if $\SS$ is clear from the context. It should be noted that for unital algebras $\L_0(\SS)=\langle
1 \rangle=\mathbb{F}$ for any $\SS$, and for non-unital algebras $\L_0 = \emptyset$.
We denote  $\L(\SS) =\bigcup\limits_{i=0}^\infty \L_i(\SS)$. 

Note that if $\SS$ is a generating set of the algebra $\A$, i.e., $\A$ is the smallest subalgebra of itself containing $\SS$, then any element of $\A$ can be expressed as a linear combination of words in elements of $\SS$.  In the above notations, if  the set $\SS$ is generating for $\A$,  then we have $\A=\L(\SS)$. If it is possible to express all elements of $\A$ using words of length at most $k$ in $\SS$, but it is not possible to use only words of length at most $k - 1$, then it is said that the length of the generating set $\SS$ is $k$.

\begin{definition}\label{alg_len}   The length $l(\A)$ of a finitely generated algebra $\A$ is the maximal length of its finite generating systems,     $l(\A)=\max \{l(\SS): \L(\SS)=\A\}$. 
\end{definition}

Notice that the unital algebra $\A$ has length   $0$ if and only if $\A = \F\cdot 1_\A$. Otherwise length is a positive integer or infinite.

The problem of the associative algebra length computation was first discussed in \cite{SpeR59, SpeR60} in the context of the mechanics of isotropic continua. Since then this important algebraic invariant became an active topic of investigations.  From one side the investigations of length function is the area with a number of interesting open algebraic problems, see \cite{GutLMSh,LafMSh,Long1,Long2,Mar09,Pap97,Paz84}. For example, even the length of the matrix algebra is not known. From another side,  the number of applications, where this invariant is used, is intensively growing nowdays, see \cite{Alp1,Alp2,BelS03,GabNRSV93,MihkSh,Ser00,FutHS17,Pea62,Laf86}.

Recent results on the lengths of non-associative algebras were obtained in the works   \cite{GutK18,GutK19, GutK21}. The key tool used across them was the notion of characteristic sequences which is defined as follows.

\begin{definition} \label{CharSeqB}\cite[Definition 2.3]{GutK21} Given a set of generators $\SS$ in an algebra $\A$, by the {\em characteristic sequence} of $\SS$
	in $\A$ we understand a monotonically non-decreasing sequence of  non-negative integers $(m_1,\ldots, m_N)$, constructed by the following rules:

\noindent 1.  Let $s_0 = \dim L_0(\SS)$. If $s_0 =1$, we set $m_1=0$.  

\noindent 2.  Denoting $s_1=\dim \L_1(\SS) - \dim L_0(\SS) $, we define ${m_{s_0+1}= \ldots  =m_{s_0+s_1}=1}$.

\noindent 3.  Let for some $r>0$, $t>1$ the elements $m_1,\ldots, m_r$ be already defined and the sets $ \L_0(\SS) ,\ldots , \L_{t-1}(\SS) $ be considered. Then we inductively continue the process in the following way. Denote $s_t=\dim \L_t(\SS) - \dim \L_{t-1}(\SS)$. Then we define $m_{r+1} =\ldots =m_{r+s_t}=t$.  
\end{definition}

\begin{lemma}\label{N=n}\cite[Lemma 2.6]{GutK21}
The characteristic sequence of $\SS$ contains exactly $\dim \A$ terms. Moreover, for the last term we have $m_N=l(\SS)$.
\end{lemma}

In the associative case the characteristic sequences are not  useful since for a given generating set $\SS$ the  dimension sequence $(\dim \L_i(\SS))_{i=1}^{\infty}$  is strictly monotone until the stabilization happens, in particular, there are no repetitions before stabilization in this sequence.  
By Definition \ref{CharSeqB}, Item 3, it means that for all $t=2,3,\ldots,$ the value $s_t\ne0$, so each integer $t$  appear in the characteristic sequence, and thus neighboring elements of characteristic sequence cannot differ in their values by more than one. In the non-associative case this is not necessarily true, and due to this reason  it was possible to produce a method of characteristic sequences   to establish  strict upper bounds on length for general non-associative algebras as well as for several known classes of  non-associative algebras, see~\cite{GutK18,GutK19}.

\begin{theorem}[{\cite[Theorem 2.7]{GutK18}}]
	Let $\A$ be a unital $\F$-algebra,  $\dim \A = n \ge 2$. Then $l( \A)\le 2^{n-2}$.
\end{theorem}

The above bound is strict, see \cite[Example 2.8]{GutK18}. The next natural step is to consider different classes of algebras determined by certain special restrictions on the growth of the length function. The first step in this direction was the notion of slowly growing length, introduced in~\cite{GutK21}.

\begin{definition}\cite[Definition 1.5]{GutK21}
We say that a class of algebras has {\em slowly growing length}, if for any representative  $\A$ of this class it holds that $l(\A) \le \dim (\A) $.
\end{definition}

In \cite{GutK21} certain general properties of this class are found.  
In particular, it is shown that this class is considerably large. Among the other classes of algebras it contains  finite dimensional associative algebras,  Lie algebras,  Leibniz algebras, Novikov algebras, and Zinbiel algebras, as well as many other important classical finite dimensional algebras. Exact upper bounds for the length of these algebras are obtained in \cite{GutK21}. 
Moreover,  some polynomial conditions on the algebra elements that guarantee the slow growth of the length function are found. 

In this paper we introduce a more general class of algebras  extending  the notion above.

\begin{definition}
We say that a class of algebras has a {\em steadily growing length}, if for every representative  $\A$ of this class it holds that $l(\A) \le c \cdot \dim (\A) + b$, where $c$ and $b$ are constants which depend only on the class. We call $c$ an {\em upper length velocity} or {\em u-velocity} of this class.
\end{definition}

By definition, classes of algebras with slowly growing length have steadily growing length and their u-velocity is 1. We show below that the opposite  is not necessarily true, cf. Example~\ref{Ex_X_d}.    

We remark that there are  
classes of algebras that are not of steadily growing length. Certain particular examples can be found  in \cite[Examples 2.8 and 5.5]{GutK18} and \cite[Example 4.4]{GutK19}.   

The main purpose of our paper is to investigate the algebras of steadily growing length.  To do this we introduce the classes of  $k$-mixing and $k$-sliding  algebras that generalize mixing and sliding algebras introduced in \cite{GutK21} for arbitrary $k$ variables.
These two properties guarantee that for a characteristic sequence $(m_1,\ldots,m_d)$ of a generating system of an algebra it holds that $m_j \le m_{j-1} + k-1$. This in turn easily provides an estimation for the length which is a linear function in the dimension with the coefficient~$k$.  

An interesting borderline case is the class of Malcev algebras introduced in 1955 by A.I. Malcev,~\cite{Mal}.
 
\begin{definition} \label{def:Mal}
	An algebra $\A$ is called a {\em Malcev algebra} if
	
	1. $xy=-yx$  for all $x,y \in \A$,
	
	2. $ (xy)(xz) =((xy)z)x + ((yz)x)x + ((zx)x)y$ for all $x,y,z \in \A$.
\end{definition}

This class is naturally connected to the class of Lie algebras via the following property.

\begin{proposition}\cite[Corollary 4.4]{Sagle}\label{prop_liemalcev}
Any two elements $a, b$ of a Malcev algebra $\A$ are contained in a Lie subalgebra of $\A$.
\end{proposition}

The detailed and self-contained exposition of Malcev algebras can be found in  \cite{Elduque,Filippov,Nagy}. This class does not fit in the framework provided by  \cite{GutK21} and the classes of mixing and sliding algebras introduced there. However the   generalization of its methods to $k$-sliding and $k$-mixing algebras allows  to prove that Malcev algebras are $3$-mixing, which guarantees that they are steadily growing. Then  it is possible to tighten the length bound even further. In this paper we prove that Malcev algebras have slowly growing length using a detailed analysis of inner structure of its words by the new method proposed in the last section of the paper. 

Our paper is organized as follows. In Section 2 we provide previously established results and notions, relevant for further proofs. In Section 3 polynomial properties which guarantee a steadily growth of length are  introduced and studied. Namely, we introduce $k$-mixing and $k$-sliding algebras, relate them with mixing and sliding algebras, and investigate their length  by means of the characteristic sequences. Section 4 contains several examples. Throughout Section 5 we examine inner structure of words in terms of lengths of their subwords and the bounds on these lengths for $k$-mixing and $k$-sliding algebras. In Section 6 we examine Malcev algebras in detail and compute the strict bound on their length.

\section{Basic results and notions}

\begin{definition}\label{Def_Irr}
	A word $w$  from a generating set $\SS$ of an algebra $\A$ is {\em irreducible}, if for each integer $m,\ 0 \leq m<l(w),$ it holds  that $w \notin L_m(\SS)$.
\end{definition}

\begin{remark}
	For all algebras $\A$ and all generating sets $\SS$ we consider $0$ to be an irreducible word of length $- \infty$.

	If $\A$ is unital we consider $1$ to be an irreducible word of length $0$. 
\end{remark}

\begin{lemma}\label{lem_1} \cite[Lemma 2.14]{GutK18}
	Any irreducible word $w$, $l(w)>1$, is a product of two irreducible words of non-zero lengths.
\end{lemma}

\begin{proposition} \label{cor_core} \cite[Corollary 2.5]{GutK21}
	\begin{enumerate}
		\item For any term $m_h$ of the characteristic sequence of $\SS$ there is an irreducible word in $\L(\SS)$ of  length $m_h$.
		\item If there is an irreducible word in $\SS$ of length $k$, then $k$ is included into the characteristic sequence of~$\SS$.
	\end{enumerate}
\end{proposition}

\begin{corollary}\label{cor_sum}
Let $\A$ be an $\F$-algebra, $\dim \A =n >2$. Assume $\SS$ is a generating set for $\A$ and $M=(m_0, m_1,\ldots, m_{n-1})$ is the characteristic sequence of $\SS$. Then for each $h$ satisfying $m_h \ge 2$ it holds that there are indices $1 \le t_1\le t_2<h$ such that $m_h=m_{t_1}+m_{t_2}$ and $m_{t_1}, m_{t_2} > 0$.

\end{corollary}

\begin{proof}
By Proposition \ref{cor_core} Item 1 each term $m_h$ of the characteristic sequence corresponds to an irreducible word of the length $m_h$. Denote it by $w_{m_h}$. By Lemma \ref{lem_1}, each irreducible word of the length $m_h \ge 2$ can be represented as a product of two irreducible words of positive lengths. Thus, $w_{m_h}=w_{k_1} w_{k_2}$ for some irreducible words $w_{k_1}, w_{k_2}$ of lengths $k_1, k_2 >0$, correspondingly. Assume $k_1 < k_2$. Then by Proposition \ref{cor_core} Item 2 there are indices $ t_1, t_2 \ge 1$ such that $m_{t_1}=k_1$ and $m_{t_2}=k_2$. Since $M$ is non-decreasing,  $1\le t_1 \le t_2 <h$.  Assume $k_1 > k_2$. Then by Proposition \ref{cor_core} Item 2 there are indices $ t_1, t_2 \ge 1$ such that $m_{t_1}=k_2$ and $m_{t_2}=k_1$. Since $M$ is non-decreasing,  $1\le t_1 \le t_2 <h$.  Assume $k_1 = k_2$. Then by Proposition \ref{cor_core} Item 2 there are indices $ t_1=t_2 \ge 1$ such that $m_{t_1}=m_{t_2}=k_1=k_2$. Since $M$ is non-decreasing,  $1\le t_1 \le t_2 <h$. In all cases, the additivity of word length gives us  $m_h=m_{t_1}+m_{t_2}$ which  concludes the proof.

\end{proof}

Below we introduce the notions of mixing and sliding algebras from our paper \cite{GutK21} which provide a large class of algebras with the slowly growing length. We  present here complete definitions to prepare the reader for the corresponding sets in $k$ variables. These sets are provided in the next section in order to determine classes of algebras having steadily growing length.

Let $x,y,z$ be variables. To introduce  the following definition we need the special  sets $Q_l$ and $Q_r$ of monomials:
$$Q_l(x,y,z) = \{ x   (z   y), x   (y   z), y   (x   z), y   (z   x), x  y, y  x, x  z, z  x, y  z, z  y,x ,y , z \}$$   
contains all  monomials of degrees 1 and 2 and only  those monomials of degree 3 in which $z$ is an argument of the first multiplication and the multiplier with $z$ is the second factor of the second multiplication.  
$$Q_r(x,y,z) = \{ (x   z)   y, (z   x)   y, (y   z)   x, (z   y)   x, x  y, y  x, x  z, z  x, y  z, z  y,x ,y , z \}$$
contains all  monomials of degrees 1 and 2 and only   those monomials of degree 3, in which $z$ is an argument of the first multiplication and the multiplier with $z$ is the first factor of the second multiplication.

\begin{definition}\label{def_2}\cite[Definition 3.1]{GutK21}
	Let $\A$ be an $\F$-algebra  such that   at least one of the following statements holds:
	
	1.  $z   (x   y) \in \langle Q_r (x,y,z) \rangle$  for all $  x,y,z \in \A$, if $\A$ is non-unital; $z   (x   y) \in \langle Q_r (x,y,z), 1 \rangle$ for all $  x,y,z \in \A$, if $\A$ is unital.
	
	2.  $(x   y)   z \in \langle Q_l (x,y,z) \rangle$ for all $  x,y,z \in \A$, if $\A$ is non-unital; $(x   y)   z \in \langle Q_l (x,y,z), 1 \rangle$ for all $  x,y,z \in \A$, if $\A$ is unital.

Then we call $\A$ a {\em sliding} algebra. 
\end{definition}

To introduce the next class of algebras we need the monomial set:  $P(x,y,z)=$ $=Q_l(x,y,z) \cup Q_r(x,y,z) =$  $$= \left \{ \begin{matrix} (x   z)   y, (z   x)   y, (y z)   x, (z   y)   x,  x   (z   y), x   (y   z), y   (x   z), y   (z   x),\\ x  y,  y  x, x  z, z  x, y  z, z  y,x ,y , z \end{matrix} \right \}, $$
i.e. we consider those monomials of degree 3 that have $z$ inside the brackets.

\begin{definition}\label{def_1}\cite[Definition 3.2]{GutK21}
	Let $\A$ be an $\F$-algebra  such that for all $  x,y,z \in \A$ it holds that $(x   y)  z,\ z    (x   y)  \in  \langle P(x,y,z), 1 \rangle$ if $\A$ is unital, and  $(x   y)  z, \ z    (x   y)  \in  \langle P(x,y,z)  \rangle$ if $\A$ is non-unital.	
	Then we call $\A$ a {\em mixing} algebra. 
\end{definition}

\begin{theorem}\label{th_sg} {\rm \cite[Theorem 3.6]{GutK21}}
	The length of a mixing or a sliding algebra $\A$ of dimension $d \ge 2$ is less than or equal to~$d$.
\end{theorem}

\section{$k$-sliding and $k$-mixing algebras}

In this section we plan to generalize Definitions \ref{def_2} and \ref{def_1} for $k$ variables.  To do this we need analogs of the sets $Q_l$ and $Q_r$ and to  introduce  them we fix the following notations.

Let  $Z=\{z_0, z_1, \ldots z_k\}$ be a certain set of variables, and $\Sigma = (z_0,\ldots, z_k)$ be  the ordered sequence of the variables from $Z$. We denote $Z_0=\{z_1, \ldots z_k\}\subset Z$. Let $T\subseteq Z$ be a certain subset. Recall that $|T|$ as usual denotes the cardinality of the set $T$. Then we define the following monomial sets.

\begin{definition} 
\begin{itemize}
\item $W(
T) $ is the set of all multilinear words in variables  from $T$ of the length $|T|$. We set $W(\emptyset)=\{1\}$ in the  unital case and  $ W(\emptyset)=\emptyset$ in the  non-unital case. 
\item 
$D(T)=\bigcup\limits_{T'\subseteq T} W(T')$ is the set of all multilinear words in variables  from~$T$. 

\item $D'(T) = D(T)  \setminus W (T )$ is the set of all multilinear words in   $T$ of the length strictly lesser than~$|T|$. Note that $D'(T)=\bigcup\limits_{T'\subsetneqq T} W(T')$.

\item $D_0 (\Sigma) = D (Z) \setminus \Bigl( z_0 W (Z_0) \cup W (Z_0)  z_0\Bigr)$ is the set of all multilinear words in   $Z$ except the words of  the length $k+1$ with $z_0$ being a factor of the last multiplication.

\item $D_l(\Sigma ) = \Bigl(\bigcup\limits_{\begin{smallmatrix}T' \subset Z_0 \\ T' \neq \emptyset \end{smallmatrix} } \Bigl\{ w_1w_2\vert w_1\in W(T'), w_2\in W(Z\setminus T')\Bigr\}\Bigr) \cup D'(Z)$ is the set of all multilinear words in $Z$ except the words of the length $k+1$ having $z_0$  in  the first factor of the last multiplication.

\item $D_r(\Sigma ) = \Bigl(\bigcup\limits_{\begin{smallmatrix}T' \subset Z_0 \\ T' \neq \emptyset \end{smallmatrix} } \Bigl\{ w_1w_2\vert w_1\in W(Z\setminus T'), w_2\in W(T') \Bigr\} \Bigr) \cup D'(Z)$ is the set of all multilinear words in  $Z$ except the words of the length $k+1$ having $z_0$  in  the second factor of the last multiplication.
\end{itemize}
\end{definition}

Now we are ready to define the classes of algebras which are crucial for our study.

\begin{definition}\label{def_2'}
Let $\A$ be an $\F$-algebra  such that   at least one of the following statements holds:
	
	1. $ W (\{ y_1, \ldots, y_k \}) x \subset \langle D_l (x,y_1,\ldots,y_k) \rangle$,

	2. $ x W (\{ y_1, \ldots, y_k \})  \subset \langle D_r(x,y_1,\ldots,y_k) \rangle$.

Then we call $\A$ a {\em $k$-sliding} algebra. 
\end{definition}

\begin{definition}\label{def_1'}
	Let $\A$ be an $\F$-algebra  such that for all $  x,y_1,\ldots,y_k \in \A$ it holds that $x W (\{ y_1, \ldots, y_k \}) \cup W (\{ y_1, \ldots, y_k \})  x \subset \langle  D_0 ( x,y_1 \ldots, y_k )    \rangle $.
	Then we call $\A$ a {\em $k$-mixing} algebra.  
\end{definition}

Let us show that these notions generalizes the notions introduced in~\cite{GutK21}.

\begin{proposition} 
Let $\A$ be a finite dimensional $\F$-algebra.

1. $\A$ is mixing if and only if $\A$ is  $2$-mixing.

2. $\A$ is sliding if and only if $\A$ is  $2$-sliding.
\end{proposition}

\begin{proof}
Below we consider only non-unital case. The  unital case can be considered in a similar way.

1. 
Assume the algebra $\A$ is mixing. Consider arbitrary $x,y_1,y_2 \in \A$. By Definition \ref{def_1} it holds that  $(y_1   y_2)  x, \ x    (y_1   y_2)  \in  \langle P(y_1,y_2,x)  \rangle$  and  $(y_2   y_1)  x, \ x    (y_2   y_1)  \in  \langle P(y_2,y_1,x) \rangle$. Since $$x W (\{ y_1, y_2 \}) \cup W (\{ y_1, y_2 \})  x = \{(y_1   y_2)  x,  x    (y_1   y_2), (y_2   y_1)  x,x    (y_2   y_1)  \}, $$ $P(y_2,y_1,x) = P(y_1,y_2,x)$ and $P(y_1,y_2,x) = D_0(x,y_1,y_2)$, we have $$x W (\{ y_1, y_2 \}) \cup W (\{ y_1, y_2 \})  x \subset \langle  D_0 ( x,y_1 y_2 )    \rangle ,$$ i.e. $\A$ is $2$-mixing.

Now let $\A$ be $2$-mixing. Consider arbitrary $x,y,z \in \A$. By Definition \ref{def_1'} it holds that  $z W (\{ x,y \}) \cup W (\{x,y \}) z \subset \langle  D_0 (\{ z,x,y\})    \rangle $. Since $\{(x   y)  z, \ z    (x   y)\} \subseteq  z W (\{ x,y \}) \cup W (\{x,y \}) z$ and $D_0 ( z,x,y)  = P(x,y,z)$, we have  $\{(x   y)  z, \ z    (x   y) \} \subseteq  \langle P(x,y,z)  \rangle$, i.e. $\A$ is mixing.

2. Assume $\A$ is sliding. If Item 1 of Definition \ref{def_2} holds for $\A$, then for arbitrary $x,y_1,y_2 \in \A$ we have $x   (y_1   y_2) \in \langle Q_r (y_1,y_2,x) \rangle$  and   $x   (y_2   y_1) \in \langle Q_r (y_2,y_1,x) \rangle$. Since  $x W (\{ y_1, y_2 \}) = \{( x (y_1   y_2),  x (y_2   y_1)  \} $,  $$Q_r(y_2,y_1,x) = Q_r(y_1,y_2,x), \mbox{ and  } Q_r(y_1,y_2,x) = D_r(x,y_1,y_2)$$ this means $x W (\{ y_1, y_2\})  \subset \langle D_r (x,y_1,y_2) \rangle$, i.e. $\A$ is $2$-sliding. The case Item 1 of Definition \ref{def_2} holds for the algebra $\A$ is similar.

If $\A$ is $2$-sliding and Item 2 of Definition \ref{def_2'} holds for $\A$, then we consider arbitrary $x,y,z \in \A$. We have $ z W (\{ x,y \})  \subset \langle D_r(z,x,y) \rangle$. Since $z  (x   y) \in z W (\{ x,y \})$ and $D_r ( z,x,y)  = Q_r(x,y,z)$, this means  $z   (x   y) \in \langle Q_r (x,y,z) \rangle$, i.e. $\A$ is sliding. The other case is similar.

\end{proof}

For a given element of a $k$-mixing algebra $$w \in x W (\{ y_1, \ldots, y_k \}) \cup   W (\{ y_1, \ldots, y_k \})  x$$ by $R(w)$ we denote the set of all words from $W(\{x,y_1,\ldots,y_k\}) $ (i.e. monomials of degree $k+1$ in variables $x, y_1,\ldots,y_k$), which are included with non-zero coefficients in at least one of the representations of $w$ as a linear combination of elements of $D_0 ( x,y_1 \ldots, y_k )$.

The following lemma is a key tool to prove the linearity of the growth length.

\begin{lemma}\label{lem_stedstep}
Let $\A$ be a $k$-mixing $\F$-algebra ($k\ge 2$), $\SS$ be its generating set and $M=(m_1,\ldots,m_d)$ be a characteristic sequence of $\SS$.  Then $m_{j+1}-m_j  \le k-1$ for all $j=1,\ldots,d-1$.  
\end{lemma}

\begin{proof}
	Assume the contrary. Let $\A$ be a $k$-mixing algebra, $\SS$ be its generating set, and assume that there exists $j$, $1\le j\le d-1$ such that the inequality $m_{j+1}-m_j  \le k-1$ does not hold. Let $p$ be the smallest index such that $m_{p+1} - m_p \ge k$. 
	
	Consider a word $w$ of length at least two. It is equal to a product $w' \cdot w''$, where $w'$ and $w''$ have non-zero lengths. We denote $s(w) = \min(l(w'),l(w''))$.

	1. Consider an irreducible word $w$ in $\SS$ of length $m_{p+1}$. Then $s(w)\ge k$. Indeed, if $s(w)\le k-1$ then  by Lemma \ref{lem_1} the word $w$ is a product of irreducible words of lengths $s(w)$ and $m_{p+1}-s(w)$. Hence by Proposition \ref{cor_core}, Item 2, there is an element equal to $m_{p+1} -s(w)$ in the characteristic sequence $M$. This is impossible, since $M$ is non-decreasing and $m_p < m_{p+1}-k+1\le m_{p+1}-s(w) $ by the assumptions.

	2.  Let us choose such an irreducible word $w_0$ of length $m_{p+1}$ that $s(w_0)$ is the smallest.  If there are several such words, we take any one of them. The chosen word is a product of two irreducible words, $w'_0$ and $w''_0$, such that $l(w'_0)=s(w_0)$. Thus we have the following 2 cases:
	
	Case 1. $w_0=w'_0 \cdot w''_0$. By Item 1 of the proof $s(w_0) \ge k$ holds. This means that $w'_0$ is a product of exactly $k$ words of positive length. Let us denote them as $w'_1,\ldots w'_k$. Consider the set $R (w'_0 \cdot w''_0$) chosen with respect to the $D_0(w''_0, w'_1,\ldots w'_k )$. At least one element $r(w)$ of this set must be an irreducible word, as words from $D_0(w''_0, w'_1,\ldots w'_k) \setminus W (\{ w''_0,w'_1, \ldots, w'_k \}) $ have length strictly less than $l(w_0)$. However, by the definition of $D_0$, $r(w)$   is a product of two words, one in $W (\{w'_i, i \in I\})$ and the other in $W (\{w''_0\} \cup \{w'_i, i \in \{1,\ldots,k \} \setminus I )\})$ for some non-empty $I \subset  \{1,\ldots,k \} $. The length of $r(w)$ is equal to $l(w_0)$, while $s(r(w))$ is equal to the length of the element in $W (\{w'_i, i \in I)$, which is strictly less than $l(w'_0) = s(w_0)$. This contradicts to the choice of $w_0$. Thus the initial assumption is false, i.e. $k$-mixing algebra cannot have a generating set with such a characteristic sequence that the difference between neighboring element is greater than~$k-1$.
	
	Case 2.  $w_0=w''_0 \cdot w'_0$.  We obtain the same contradiction similarly, considering an irreducible element of  $R (w''_0 \cdot w'_0$) chosen with respect to the  set $D_0(w''_0, w'_1,\ldots w'_k )$
\end{proof}

For a given element of a $k$-sliding algebra $w \in  W (\{ y_1, \ldots, y_k \})  x$ (or $w \in W (\{ y_1, \ldots, y_k \})  x$) by $U_l(w)$ (or $U_r(w)$) we denote the set of all words from $W(\{x,y_1,\ldots,y_k\}) $  which are included with non-zero coefficients in at least one representation of $w$ as a linear combination of elements of   $ D_l (x,y_1,\ldots,y_k) $ (or $  D_r (x,y_1,\ldots,y_k) $).

\begin{lemma}\label{lem_stedstepZ}
Let $\A$ be a $k$-sliding $\F$-algebra ($k\ge 2$), $\SS$ be its generating set and $M=(m_1,\ldots,m_d)$ be a characteristic sequence of $\SS$.  Then $m_{j+1}-m_j  \le k-1$ for all $j=1,\ldots,d-1$.  
\end{lemma}
\begin{proof}

	Assume the contrary: let $\A$ be a $k$-sliding algebra satisfying Item 2 of Definition \ref{def_2'} (the case of Item 1 can be proven similarly) and let $\SS$ be a generating set of $\A$ such that for its characteristic sequence $M$ the inequality $m_{j+1}-m_j  \le k-1$ does not hold for all $j=1,\ldots,d-1$. Let $k$ be the smallest index such that $m_{p+1} - m_p \ge k$. 

	Consider a word $w$ of length at least two. It is equal to a product $w' \cdot w''$, where $w'$ and $w''$ have non-zero lengths. We denote $l_r(w) =l(w'')$.

	Let us choose such an irreducible word $w_0$ of length $m_{p+1}$ in $\SS$ that  $l_r(w)$ is minimal (if there are multiple possible candidates, we can choose one at random). By Lemma \ref{lem_1} $w_0$ is equal to $w'_0 \cdot w''_0$, where both factors have non-zero length.
	
	1. $l_r(w_0)=l(w''_0)\le k-1 $,  cannot hold.  This would mean that $l(w'_0)=m_{p+1}-l(w''_0)$ which by Proposition \ref{cor_core}, Item 2 would mean that there is an element of characteristic sequence equal to $m_{p+1} -l(w''_0)$, and that is impossible: $M$ is non-decreasing and $m_p < m_{p+1}-k+1 \le m_{p+1} -l(w''_0)$.
	
	2. If $l(w''_0) >k-1$,  $w''_0$ is equal to a product of exactly $k$ words $w''_1, \ldots, w''_k$ of positive length. Consider the set  $D_r(w'_0, w''_1,\ldots w''_k )$. As all of the elements of $D_r(w'_0, w''_1,\ldots w''_k ) \setminus W (\{ w'_0,w''_1, \ldots, w''_k \}) $ have lengths strictly less than $l(w_0)$, at least  one element of $U_r(w_0)$ must be irreducilble, since otherwise $w_0$ is a linear combination of shorter words or reducible words of the same lengths, i.e. reducible as well. Let us denote this element by $u(w_0)$. Note that $l(u(w_0)) = l(w_0)$, while the second factor in the last multiplication of $u(w_0)$ has length strictly less than $l(w''_0)$, which contradicts the choice of $w_0$. Thus the initial assumption is false.
\end{proof}

\begin{theorem}\label{th_stgr}
Let $\A$, $\dim \A=d\ge 2$,  be a $k$-mixing or a $k$-sliding algebra, $k \ge 2$. Then the length of $\A$ has steady growth and  $(k-1)$ is its u-velocity. 
\end{theorem}
\begin{proof}
Follows directly from Lemma \ref{lem_stedstep} for $k$-mixing algebras or from Lemma \ref{lem_stedstepZ} for $k$-sliding algebras. Namely, for a generating set $\SS$ of $\A$ with $l(\SS)=l(\A)$ and characteristic sequence $(m_1,\ldots,m_d)$ of $\SS$ we have $m_1\le 1$ and $l(\SS) =m_d \le m_{d-1} + k-1 \le \ldots \le m_1 + (d-1)(k-1) \le 1+ (d-1)(k-1) \le d (k-1)$.
\end{proof}

\section{Examples}

Below we present two classes of algebras satisfying, respectively, $k$-sliding and $k$-mixing properties such that $k-1$ is their minimal u-velocity.  We call  them $k$-round and $k$-based algebras, respectively. Some examples of $k$-round algebras appeared in~\cite{GutK21} as examples of algebras that have linear in dimension growth of the length but are not of slowly growing length, cf.~\cite[Proposition 4.25]{GutK21}.

\begin{definition}\label{k-round}
We say that an algebra $\A$ is {\em $k$-round} ($k\ge 2$) if for all $  x,y_1,\ldots,y_k \in \A$ and for any product $v=y_1\cdots y_k$ with any placement of parentheses it   holds that $x v = 0$.
\end{definition}

\begin{proposition}
Any $k$-round algebra is $k$-sliding.
\end{proposition}
\begin{proof}
Follows immediately from Definition \ref{def_2'}, Item 2.
\end{proof}

An example of a $k$-round algebra is constructed below. In order to check polynomial identities for algebras we   use the following lemma.

\begin{lemma}\cite[Lemma 4.1]{GutK21}\label{lem_distr}
	Consider a finite-dimensional algebra $\A$ over the field $\F$, its basis $\{e_1,\ldots,e_d\}$ and multilinear function $G$ of $k$ arguments such that $G(e_{i_1}, \ldots, e_{i_k}) = 0$ for all $i_t \in \{1,\ldots,d\}$. Then for all $a_1,\ldots,a_k \in \A$ it holds that $G(a_1, \ldots, a_k) =0$.
\end{lemma} 

\begin{example}\cite[Example 4.26]{GutK21}
Consider an algebra $\mathcal{E}_d$ with the basis $x_1,\ldots,x_d$, $d \ge k \ge 2$ and the following multiplication law:
		$$x_j x_1 = x_{j+1}, \ j=1,\ldots,k-2,$$ 
		$$x_i x_{k-1} = x_{i+1}, \ i=k-1,\ldots, d-1,$$
	with other products being zero. We have $$l(\mathcal{E}_d) \ge l(\{x_1\}) = (k-1) d - (k-2)(k-1) .$$
	
Let us prove that $\mathcal{E}_d$ satisfies Definition~\ref{k-round}. 

At first, we consider $x,y_1,\ldots,y_k \in \{x_1,\ldots, x_d\}$. A word $x v$,  where $v$ is a product of $y_1,\ldots,y_k$, is indeed zero as $v$ cannot be neither $x_1$ nor $x_{k-1}$. So, the required condition holds for the basis of $\mathcal{E}_d$. Then  by Lemma \ref{lem_distr} it is satisfied for other elements as well. 
\end{example}

\begin{proposition}
Let $k\ge 2$ be integer. The minimal u-velocity for the class of $k$-round algebras is $k-1$.

\end{proposition}
\begin{proof}
Note that $k$-round algebras are $k$-sliding and by Theorem \ref{th_stgr} their u-velocity is~$k-1$.

Consider a number $c <k-1$ and an arbitrary   $b\in \R$. There exists an integer $d\ge k$ such  that $c d + b < (k-1) d - (k-2)(k-1)$. Since $\mathcal{E}_d$ belongs to the class of $k$-round algebras  and $(k-1) d - (k-2)(k-1) \le l(\mathcal{E}_d)$, this means that the inequality $l(\A) \le c \cdot \dim \A +b $ does not holds universally for  the members of this class. Since we can pick $b$ arbitrarily, this means that $c$ is not a u-velocity of the class of $k$-round algebras.

Thus, $k-1$ is indeed a minimal u-velocity of the class of $k$-round algebras.
\end{proof}

\begin{definition}\label{k-based}
We say that an algebra $\A$ is {\em $k$-based}, $k\ge 2$,  if for all $  x,y_1,\ldots,y_k \in \A$ and any fixed placement of parentheses $u v = y_1\cdots y_k$ it holds that $x (u v) = u (x v)$, $(u v) x = (u x) v$.
\end{definition}

\begin{proposition}
A $k$-based algebra is $k$-mixing.
\end{proposition}
\begin{proof}
Follows immediately from Definition \ref{def_1'}.
\end{proof}

The following is an example of a $k$-based algebra. 

\begin{example} \label{Ex_X_d}
Consider the algebra $\mathcal{X}_d$ with the basis $x_1,\ldots,x_d$, $d \ge k \ge 2$ and the following multiplication law:
		$$x_1 x_j = x_{j+1}, \ j=1,\ldots,k-2,$$ 
		$$x_{k-1} x_{i} = x_{i+1}, \ i=k-1,\ldots, d-1,$$
	with other products being zero. We have $$l(\mathcal{X}_d) \ge l(\{x_1\}) = (k-1) d - (k-2)(k-1) .$$

	Let us prove that $\mathcal{X}_d$ is indeed a $k$-based algebra, namely, $\mathcal{X}_d$ satisfies Definition~\ref{k-based}. 
	
	We consider $x,y_1,\ldots,y_k \in \{x_1,\ldots, x_d\}$. If the required condition holds for the basis of $\mathcal{X}_d$, then  by Lemma \ref{lem_distr} it is satisfied for other elements as well. 

	 For the word $x (u v)$, where $u v = y_1\cdots y_k$ with a certain fixed placement of parentheses, there are three possibilities:

\begin{itemize}

	\item Assume $u = x_{j}$, $j <k-1$. This means $u v = 0$ and  $u(x v)= 0$ by the multiplication laws, from which it follows  that $ x (uv) = 0 = u (xv)$.

	\item Assume $u = x_{k-1}$. In this case the following three possibilities appear:
\begin{itemize}
	
	\item If $x = x_j$, $j <k-1$, then similarly to Case 1 both  $ x (uv)$ and $u (xv)$ are equal to zero.

	\item If $x = x_{k-1}$, then $u=x$ and the equality of  $ x (uv)$ and $u (xv)$ is trivial.

\item If $x=x_{i}$, $i > k-1$, then  $x (u v) = 0$ and  $x v= 0$ by the multiplication laws, from which follows $ x (uv) = 0 = u (xv)$. \end{itemize}

	\item Assume $u = x_i$, $i > k-1$. This means that $u v = 0$ and  $u(x v)= 0$ by the multiplication laws, from which follows $ x (uv) = 0 = u (xv)$. 

\end{itemize}

Also for the word $(u v) x$, where $u v = y_1\cdots y_k$ with a certain fixed placement of parentheses, there are three possibilities:

\begin{itemize}

	\item Assume $u = x_{j}$, $j <k-1$. This means $u v = 0$ and  $u x= 0$ by the multiplication laws, from which follows $(u v) x = 0 = (u x) v$.

	\item Assume $u = x_{k-1}$. This means $u v$ and $u x$ are either zero or some $x_i$, $i > k-1$ which leads by multiplication laws to $(u v) x = 0 = (u x) v$.

	\item Assume $u = x_{i}$, $i >k-1$. This means $u v = 0$ and  $u x= 0$ by the multiplication laws, from which follows $(u v) x = 0 = (u x) v$.

\end{itemize}
This concludes the proof that  $\mathcal{X}_d$ is indeed a $k$-based algebra.
\end{example}

\begin{proposition}
The minimal u-velocity for the class of $k$-based algebras is~${(k-1)}$.
\end{proposition}
\begin{proof}
Note that $k$-based algebras are $k$-mixing and by Theorem \ref{th_stgr} $k-1$ is a u-velocity for this class.

Consider a number $c <k-1$ and an arbitrary number $b\in \R$. There exists a natural number $d\ge k$ such  that $c d + b < (k-1) d - (k-2)(k-1)$. Since $\mathcal{X}_d$ belongs to the class of $k$-based algebras and $(k-1) d - (k-2)(k-1) \le l(\mathcal{X}_d)$, this means that the inequality $l(\A) \le c \cdot \dim \A +b $ does not holds universally for the members of this class. Since we can pick $b$ arbitrarily, this means that $c$ is not a u-velocity of the class of $k$-based algebras.

Thus, $k-1$ is indeed a minimal u-velocity of the class of $k$-based algebras.
\end{proof}

The following class of algebras shows that there are algebras of  steadily growing lengths that are neither mixing nor sliding, so these classes of algebras provide a useful tool to investigate algebra length. The general question of characterization for these algebras remains open. 

\begin{definition}\label{def_vinberg}  \cite{Car}
	An algebra $\A$ is called a {\em Vinberg algebra} if $$ (xy)z-x(yz) = (xz)y-x(zy)$$ for all $x,y,z \in \A$.
\end{definition}

These algebras are sometimes called right-symmetric algebras.

\begin{example}
	Consider an algebra $\mathcal{V}_d$ with the basis $x_1,\ldots,x_d$, $d \ge 4$ and the following multiplication law:
	$$x_i x_j = x_{i+j}, i+j \le d-1,$$ 
	$$x_{d-1} x_{d-2} = x_d,$$
	with other products being zero.
	
	It is easy to see that $\mathcal{V}_d$ is neither $k$-mixing nor $k$-sliding for any  $k\in \{2,\ldots, d-2\}$. The characteristic sequence of the set $\{x_1\}$ is $(1,2,\ldots,d-2,d-1,2d-3)$ and $(2d-3)  -(d-1) = d-2$.
	
	However, $\{ \mathcal{V}_d \}$ is a class of steadily growing algebras with u-velocity $2$ since it is straightforward to check that $l(\mathcal{V}_d) = 2d-3$. 
	
\end{example}

\section{Sprout sequences}

The aim of this section is   to establish new methods to investigate  the inner structure of words in algebras which allow to compute the length. The definitions below are the first steps in this direction.

\begin{definition}
Consider an algebra $\A$, its generating set $\SS$ and a word $w$ in $\SS$. We say that a word $w'$ in $\SS$ is a {\em subword} of $w$ if it is a factor of one of the multiplication inside of $w$ or it is $w$ itself. We also say that $w'$ is a {\em proper subword} of $w$ if it is a subword which is not $w$ itself.
\end{definition}

\begin{example}
Let $\A$ be an algebra and $\SS$ be its generating set such that there exist $x,y \in \SS$ with $x\neq y$. For a word $w = (x y) (y (x x))$ in $\SS$, words $x, y, xx, xy, (y (x x)), (x y) (y (x x))$ are its subwords, while $yy$ or $yx$ are not. Note that there is no subword of length 4 in~$w$.
\end{example}

\begin{lemma}\label{lem_subduh}
Let $\A$ be an algebra, $\SS$ be its generating set and $w$ be a word in $\SS$ such that $w = u v$ where $u,v$ are also words in $\SS$ of positive length. Then the set of proper subwords of $w$ is the union of the sets of subwords of $u$ and~$v$.
\end{lemma}
\begin{proof}
It is straightforward to see that a subword of $u$ or $v$ is also a subword of $w$, moreover, a proper subword, as it is shorter than~$w$.

On the other hand, a proper subword of $w$ is a factor in one of the multiplications inside of $w$. If it is a factor of the last multiplication, then it is either $u$ or $v$. Otherwise, it is a factor in one  of the multiplications inside of $u$ or $v$. In both cases it is a subword of $u$ or~$v$.
\end{proof}

\begin{definition} \label{def_sprout} Let $\A$ be an $\F$-algebra, $\SS$ be its generating set and $w$ be a word of length at least 2 in  $\SS$. We  construct inductively two sequences $(w_1, \ldots, w_r)$ and $(w_1',\ldots, w'_r)$ of  subwords  of $w$ as follows:

\begin{enumerate}

\item Set $w_1$ and $w'_1$ to be the two factors of the last multiplication in $w$ (i.e. $w = w_1 w'_1$ or $w = w'_1 w_1$) such that $l(w_1)\le l(w'_1)$. 

\item Assume $w_1, \ldots, w_{j-1}$ and $w'_1, \ldots, w'_{j-1}$ are constructed and $l(w'_{j-1}) \ge 2$.  Set $w_j$ and $w'_j$ to be the two factors of the last multiplication in $w'_{j-1}$  such that $l(w_j)\le l(w'_j)$. 

\item The process stops if $l(w'_r) = l(w_r) = 1$ for a certain $r$.
\end{enumerate}	

We call the   sequence $(w_1, \ldots, w_r)$ a {\em sprout sequence}, or just a {\em sprout}, of $w$ and the   sequence $(w'_1, \ldots, w'_r)$  a {\em supporting sequence} of~$w$.

If $(w_1, \ldots, w_r)$ is a sprout sequence, we say that $(l(w_1), \ldots, l(w_r))$ is an {\em $l$-sprout sequence}.

\end{definition}

Let us start to analyze sprouts and their properties with an example.
\begin{example}
Let $\A$ be an $\F$-algebra, $\SS$ be its generating set and $x,y,z \in \SS$. Then both sequences
$\Bigl((x y) (z y), y, z, x\Bigr)$ and 
$\Bigl(( (x z) z) y, xy, z\Bigr) $
are sprout sequences of the word $w=((x y) (z y)) (( (x z) z) y)$ and $(4,1,1,1)$ and $(4,2,1)$ are  $l$-sprout sequences of~$w$.
\end{example}

The supporting sequence has the following useful property.

\begin{lemma}\label{lem_ind}
Let $\A$ be an $\F$-algebra, $\SS$ be its generating set, $w$ be  a word of length at least 2  in $\SS$, $(w_1, \ldots, w_r)$ and $(w_1',\ldots, w'_r)$ be sprout and supporting sequences of $w$, and $j$ be an integer, $2\le j \le r$. Then the sequence $(w_j,\ldots, w_r)$ is a sprout sequence of $w'_{j-1}$.
\end{lemma}
\begin{proof}
Follows directly from Definition~\ref{def_sprout}.
\end{proof}

The next statement is converse to Lemma~\ref{lem_ind}.

\begin{lemma}\label{lem_seqexp}
Let $\A$ be an $\F$-algebra, $\SS$ be its generating set, $w,w'$ be  words  in $\SS$, $l(w) \ge l(w')> 0$. We denote by  $(w_1,\ldots, w_r)$  a sprout sequence of $w$. Then $(w',w_1,\ldots, w_r)$ is a sprout sequence of both $w w'$ and $w' w$. Additionally, $(l(w'),l(w_1),\ldots, l(w_r))$ is an $l$-sprout sequence for both words $w w'$ and~$w' w$.
\end{lemma}
\begin{proof}
Follows directly from Definition~\ref{def_sprout}.
\end{proof}

The following concept further expands the applications of the idea of the sprout sequence.  Namely, sprout sequence consists of  'shorter' subwords of a given word. In the next definition we directly  restrict  their lengths.

\begin{definition}\label{def_bound}
We say that a word $w$ is {\em $k$-bounded},  $k \ge 2$,  if there exists  an $l$-sprout sequence for $w$ that does not contain elements of the length greater than $k$ or the word itself has length 1. We call the respective sequence {\em $k$-bounded} as well.
\end{definition}

\begin{lemma}\label{lem_smolb}
Let $\A$ be an $\F$-algebra, $\SS$ be its generating set, and $k\ge 2$ be an integer. Assume that $w$ is  a word  in $\SS$ such that  $l(w) \le  2k-1$. Then $w$ is $k$-bounded.  
\end{lemma}
\begin{proof}
If $w$ has length 1, then the statement is evident. Otherwise it is enough to note that the elements of its sprout sequence are the shorter factors of last multiplication in the respective subwords of $w$, which means that their lengths are less than or equal to $\frac{l(w)}{2} \le \frac{2k-1}{2} <k$.
\end{proof}

\begin{lemma}\label{lem_subbound}
Let $\A$ be an $\F$-algebra, $\SS$ be its generating set, and $k\ge 2$ be an integer. Assume, $w$ be  a $k$-bounded word in $\SS$. Then every subword of $w$ is also $k$-bounded.
\end{lemma}
\begin{proof}
We prove this statement using induction on the length $l$ of the word~$w$.

The base. For $l=1,\ldots,2k-1$ the statement is evident by Lemma~\ref{lem_smolb}.

The step. Assume that the statement holds for $l=1,\ldots,t$ and consider $l=t+1$, $t \ge 2k-1$. There are two possibilities for a subword~$w'$: either it is a factor of $w$ or it is a subword of one of factors of~$w$. Let us consider each of these cases separately:

a.  Let $w = w' w''$ or $w = w'' w'$. If $l(w')\le l(w'')$, then $l(w') <k$ since it is a first element in every $l$-sprout sequence of a $k$-bounded word. If $l(w') > l(w'')$, we consider a sprout sequence of $w$ corresponding to the $k$-bounded $l$-sprout sequence. By Lemma \ref{lem_ind} it has the form $(w'', T)$, where $T$ is a sprout sequence of $w'$. Thus, the sequence of lengths of elements of $T$ is $k$-bounded itself as a subsequence of a $k$-bounded sequence.   

b. Let $w'$ be not a factor of $w$, i.e. a proper subword of one of factors of $w$. Note that factors of the last multiplication are $k$-bounded by Item a. of this lemma and shorter than $w$. Thus $w'$ is $k$-bounded by the induction hypothesis.
\end{proof}

Now we connect the notion of the sprout sequences to the previously introduced classes of $k$-mixing and $k$-sliding algebras, namely, we are going to prove that in these classes of algebras any irreducible word $w$  can be represented as a linear combination of $k$-bounded words of the length not greater than~$l(w)$.

\begin{proposition}\label{prop_repres}
Let  $k\ge 2$ be an integer, $\A$ be a $k$-mixing $\F$-algebra, and $\SS$ be its generating set. Assume that  $w$ is an irreducible word in $\SS$, $l(w) \ge 2$. Then there exist irreducible $k$-bounded words $w_1,\ldots,w_n$ such that $l(w_i) \le l(w)$ and $w \in \langle w_1,\ldots,w_n \rangle$.
\end{proposition}

\begin{proof}
We will prove this statement by induction on the length $l$ of the word $w$.

The base for $l$. For $l=2,\ldots,2k-1$ the statement is trivial, as $w$ itself is $k$-bounded by Lemma \ref{lem_smolb}.

The step for $l$. Assume that the statement holds for $l=2,\ldots,r$ where $r\ge 2k-1$ and consider an irreducible word $w$ of length $r+1$. By Lemma \ref{lem_1} we can write $w =v \cdot u$, where $v$ and $u$ are irreducible words of non-zero lengths.

Now we need to start another induction based on the parameter $j = s(w) =\min (l (u), l(v))$ in order to use the $k$-mixing property of $\A$.

1. The base for $j$.  Assume $j=1,\ldots,k-1$. There are two possibilities.

1.1. Let  $l(v) = j$, i.e. $l(u) \ge l(v)$ and $l(v) <k$. Note that $l(u) \ge \frac{r+1}{2} \ge k \ge 2$, which means we can apply the induction hypothesis for $l$ to $u$. Let $u = f_1 u_1 + \ldots + f_n u_n$ be the resulting decomposition, where $f_i$ are scalar coefficients from $\F$, and $u_i$ are $k$-bounded irreducible words. This allows us to write $w = v u =  f_1 (v u_1) + \ldots + f_n (v u_n)$. 

If $l(u_i) \ge k$, then by Lemma \ref{lem_seqexp} there is a sprout sequence of $v u_i$ of the form $(v, T)$, where $T$ is a $k$-bounded sprout sequence of $u_i$. This means that $v u_i$ is $k$-bounded as $l(v) <k$ and length of every element of $T$ is less than $k$ as well.

If $l(u_i) \le k-1$, then $v u_i$ is $k$-bounded by Lemma~\ref{lem_smolb}.

Thus, $w =   f_1 (v u_1) + \ldots + f_n(v u_n)$ is a representation of the initial word as a linear combination of $k$-bounded words. We will prove that each element $w' = (u v_i)$ in this combination can be reduced to the linear combination of $k$-bounded irreducible words.

1.1.a. If $w'$ is reducible, we can decompose it as a linear combination of irreducible words of lengths less than $r$ and apply the induction hypothesis for $l$ to them.

1.1.b. If $w'$ is irreducible, it is enough to note that it is already $k$-bounded.

1.2. The case $l(u)=j$ is similar, which concludes the proof of the base for~$j$.

2. The step for $j$. Assume that the statement holds for values of $j=1,\ldots,p$ with $p\ge k-1$. For $j=p+1$ there are again two possibilities. 

2.1. Let  $l(v)=p+1$. Then it holds that $p+1 \ge k$, which means that $v$ is equal to a product of exactly $k$ words of non-zero length. Let us denote them as $v_1,\ldots v_k$. Consider the representation of $w = v \cdot u$ as a linear combination of elements of $D_0(u, v_1,\ldots v_k )$. We will prove that each element $w'$ in this combination can be reduced to a linear combination of $k$-bounded irreducible words.

2.1.a. If $w'$ is reducible, we can decompose it as a linear combination of irreducible words of lengths less than $r$ and apply the induction hypothesis for $l$ to them.

2.1.b. If $w'$ is irreducible and $l(w) \le r$, then we can apply the induction hypothesis for $l$ to $w'$ directly.

2.1.c. If $w'$ is irreducible and  $l(w)=r+1$, i.e. an element of $R (v \cdot u$), then $s(w') < s(w)$ by the definition of this set and the induction hypothesis for $j$ can be applied.

2.2. $l(u)=p+1$ is similar.  

This concludes the induction by $j$ and thus by induction by $l$ concludes the proof.
\end{proof}

\begin{proposition}\label{prop_represZ}
Let  $k\ge 2$ be an integer, $\A$ be a $k$-sliding $\F$-algebra,  $\SS$ be its generating set. Assume that   $w$ is an irreducible word in $\SS$, $l(w) \ge  2$. Then there exist irreducible $k$-bounded words $w_1,\ldots,w_n$ such that $l(w_i) \le l(w)$ for each $i=1,\ldots, n$, and $w \in \langle w_1,\ldots,w_n \rangle$.
\end{proposition}
\begin{proof}
This statement can be demonstrated similarly to the proposition above.
\end{proof}

\begin{corollary}\label{cor_exbound}
Let $k\ge 2$ be an integer, $\A$ be a $k$-sliding or a $k$-mixing $\F$-algebra,  $\SS$ be its generating set. Assume that   $w$ is an irreducible word in $\SS$, $l(w) \ge  2$. Then there exists a $k$-bounded irreducible word in $\A$ of the   length~$l(w)$. 
\end{corollary}
\begin{proof}
By successive application of Propositions \ref{prop_repres} and \ref{prop_represZ}  the word $w$ belongs to $  \langle w_1,\ldots,w_n \rangle$, where every $w_i$ is a $k$-bounded irreducible word. Since $w$ is irreducible as well, there exists an index  $i$ such that  $l(w_i) = l(w)$, as otherwise $w$ would be equal to a linear combination of shorter words. 
\end{proof}

\section{Malcev algebras}

The following definition of Malcev algebras in the multilinear form  is equivalent to Definition~\ref{def:Mal} and are already given, see~\cite{Sagle}.

\begin{proposition}\cite[Propostion 2.21]{Sagle}\label{prop_prop}.
An algebra $\A$ of characteristic not 2 is a Malcev algebra
if and only if $\A$ satisfies the identities:

1. $xy= -yx$ and 

2. $(xy)( zw )= x((wy)z) + w((yz) x)+y((zx) w)
+ z((xw) y)$ 

for all $x, y, z, w \in \A$.
\end{proposition}

Now we provide two consequences of this representation that we need in our further considerations.

\begin{lemma}\label{lem_smolswap}
	Let $\A$ be a Malcev algebra over a field $\F$ with $\char \F \neq 2$ and $a,b,c,w'$ be its arbitrary elements. It holds that 
	
	\begin{align}\label{eq_smolswap}
		(a b) (c w') = -(ac)(b w')  +  a((w'b)  c) + b ((ca)  w') + c ((aw')  b) + \notag \\
		+  a ((w'c)  b) + c ((ba) w') + b ((aw') c).
	\end{align}
	
\end{lemma}
\begin{proof}
	By Proposition \ref{prop_prop} we can write $(ab)(c w') = a((w'b) c) + w' ((bc)  a) + b ((ca) w') + c ((aw') b)$. Note that $w' ((bc) a) = - w' ((cb) a)$ and by the same proposition  $w' ((cb) a) = (ac)(b w') - a ((w'c) b) - c ((ba) w') - b ((aw') c)$. By substituting this result in the first equality we obtain the desired identity.
\end{proof}

\begin{lemma}\label{lem_swap}
	Let $\A$ be a Malcev algebra over a field $\F$ with $\char \F \neq 2$ and $a,b,c,d,w'$ be its arbitrary elements. Then we have
	\begin{align}\label{eq_swap}
		(a b) ((cd) w') = - (c d) ((ab)w')  + ((cw') b) (d a) -  a ((bd) (cw'))  -\notag \\ - b((d(cw'))  a)- d ((w' (ab))c) - (ab) ((dw')c) - c ((d b) (a w') )  + \notag \\ + c( d ((w' a) b)) +c ( b((ad)  w')) + c( a((dw') b))- c ((d w') (ab)).
	\end{align}
	
\end{lemma}

\begin{proof}
	By Proposition \ref{prop_prop} we can write 
	\begin{equation}\label{eq:6.3} (ab) ((cd) w') = (d (ab)) (c w') - d ((w' (ab))  c) - w' (((ab) c) d) - c ((d w')(ab)).\end{equation}  Now we consider  the words on the right hand side of this equality.
	
	Note that the first word $(d (ab)) (c w') = (c w') ((ab)  d)$. By the same proposition,  $ (c w') ((ab)  d) = ((cw') b) (d a) -  a ((bd)  (cw')) - b((d(cw')) a) - d ((cw')  (ab)) $.  
	
	The third word $w' (((ab) c) d) = - w' ((c (ab)) d)$ and by the same proposition   $w' ((c (ab)) d) = (d c) ((ab) w') - d((w' c) (ab)) - c (((ab) d) w') - (ab) ((dw') c)$.

	Now, $((ab) d)  w' = w' ((b a)  d)$. By the same proposition, $w' ((b a)  d) = (d b) (a w') - d ((w' a) b) - b((ad) w') - a((dw') b)$.
	
	By substituting these results in the expression \eqref{eq:6.3}  and noting that  $(c d) ((ab)w') = -(d c) ((ab) w')$ we obtain the desired identity.
\end{proof}

\begin{corollary}
A Malcev algebra over a field $\F$ with $\char \F \neq 2$ is 3-mixing and has steadily growing length with u-velocity~2.
\end{corollary}

\begin{proof}
	By Proposition \ref{prop_prop} a Malcev algebra $\A$ over a field $\F$ with $\char \F \neq 2$ satisfies
 $$(xy)( zw) = x((wy) z) + w((yz) x)+y((zx) w) + z((xw) y)$$  for all $x, y, z, w \in \A$. By rearranging words we achieve
\begin{equation}\label{eq_malcev}
x((wy) z) = (xy)  (zw) -   w((yz)  x)- y((zx) w) - z((xw) y)  \in D_0 (x,y,z,w).
\end{equation}
Since $\A$ is anticommutative the inclusion \eqref{eq_malcev} implies
$$x (z  (wy)),\ ((wy) z)x, \ (z (wy))x \in \langle D_0 (x,y,z,w) \rangle$$  for all $x, y, z, w \in \A$. Thus, we can conclude that  $\A$ is $3$-mixing by Definition \ref{def_1'}, which in turns leads to the class having  steadily growing length with u-velocity 2 by Theorem~ \ref{th_stgr}.
\end{proof}

The above result can be substantially improved using the notions introduced in Section 5. 
Namely, our next goal is to demonstrate that Malcev algebras have a slowly growing length. 

In the paper \cite{GutK21} we proved that 2-mixing and 2-sliding algebras have slowly growing length. This was done by the direct application of  a particular case of   
Theorem \ref{th_stgr} for $k=2$, see \cite[Theorem 3.6]{GutK21}. 
However, since an arbitrary Malcev algebra $\A$ is  $3$-mixing, we cannot guarantee that $m_{j} - m_{j-1} \le 1$ for the characteristic sequence $(m_1,\ldots,m_d)$ of a generating set $\SS$ of $\A$.  Indeed, Theorem \ref{th_stgr} implies just that the inequality $m_{j} - m_{j-1} \le 2$ for each $j=1,\ldots, d$.  We are going to demonstrate that in a Malcev algebra if for some $l, 1\le l\le d$ the equality $m_l - m_{l-1} = 2$ holds, then $m_{l-1}=m_{l-2}$. This implication provides the bound  $l(\SS) = m_d \le  m_1 + (d-1) \le d = \dim \A$ since each addition of 2 is preceded by the addition of~0. 

To  realize this plan  we need the following definition.

\begin{definition}
Consider an algebra $\A$ over a field $\F$ and its generating set $\SS$. We say that the words $u$ and $v$ in $\SS$ are {\em equivalent} if $u,v$ are linearly dependent modulo $\L_{\max (l(u),l(v))-1}(\SS)$, i.e., a certain linear combination of $u$ and $v$ belongs to the subspace $\L_{\max (l(u),l(v))-1}(\SS)$. We denote this as~$u \sim v$.
\end{definition}

For the relation $\sim$ we can prove several useful properties.

\begin{lemma}\label{lem_sim}
\begin{enumerate}
\item $\sim$ is an equivalence relation.
\item The word $u$ is reducible if and only if $u \sim 0$.
\item If $u$ is irreducible and $u \sim v$, then $v$ is also irreducible.
\item If $u\sim v$ and $w$ is a word in $\SS$ then $uw \sim vw$ and $w u \sim w v$.
\end{enumerate}
\end{lemma}
\begin{proof}
\begin{enumerate}
\item Reflexivity and symmetry of $\sim$ are evident, and it can be easily checked that it satisfies transitivity as $u,v$ linearly dependent modulo $\L_{h}$ implies that they are linearly dependent modulo $\L_{h'}$, where $h \le h'$.
\item Follows directly from Definition \ref{Def_Irr}.
\item Follows directly from the first two items.
\item If  $u \sim v $, i.e. $f_u u + f_v v + x =0$ where $f_u$ and $f_v$ belong to $\F$ and not simultaneously zero, while $x \in \L_{\max (l(u),l(v))-1}(\SS)$, then $f_u (u w) + f_v (v w) + x w =0$ and $f_u (w u) + f_v (w v) + w x =0$. Since $x w$ and $w x$ belong to $\L_{\max (l(u),l(v))-1 + l(w)} = \L_{\max (l(u)+l(w),l(v)+l(w))-1}$, this implies $uw \sim vw$ and $w u \sim w v$.
\end{enumerate}
\end{proof}

Thus, the equivalent words behave   similarly in terms of being irreducible and generating new irreducible words.

\begin{lemma}\label{lem_same}
Let $\A$ be a finite dimensional $\F$-algebra, $\SS$ be its generating set, and $M=(m_1,\ldots,m_d)$ be the characteristic sequence of $\SS$. The equality  $m_{j-1}=m_{j-2}$ holds  if and only if there exist at least two non-equivalent irreducible words of length $m_{j-1}$ in $\SS$.
\end{lemma}
\begin{proof}
By the definition of $M$, there are exactly $\dim \L_{k}(\SS) - \dim \L_{k-1}(\SS)$ elements in $M$ that are equal to a given $k$. The condition  $m_{j-1}=m_{j-2}$ means that there are at least two equal values in the characteristic sequence, so it is equivalent to $\dim \L_{m_{j-1}}(\SS) - \dim \L_{m_{j-1}-1}(\SS) \ge 2$. The last condition  means  the existence of at least two non-equivalent irreducible words of length $m_{j-1}$ in~$\SS$.
\end{proof}

Now we introduce two concepts specific to the $3$-mixing and $3$-sliding cases.

\begin{definition}\label{def_step}
Consider a $3$-mixing or a $3$-sliding algebra $\A$ and its generating set $\SS$. For a 3-bounded word $w$ in $\SS$, $l(w)\ge 2,$  we define the {\em step function} $$\sigma(w) = \min \{t \ : \ w \mbox{ has a subword of length } l-2t-1\}.$$
\end{definition}

\begin{lemma}  \label{Rem_step}
The notion of step function is  defined correctly. \end{lemma}
\begin{proof}  For a 3-bounded word $w$ by Definition \ref{def_bound} there exists an $l$-sprout sequence consisting only of $1$ and $2$, meaning that $w$ can be factored into a product of words of lengths $1$ and $2$. If the last multiplication in $w$ has a factor of length 1, then $w$ has a subword of length $l(w) -1 = l(w) - 2\cdot 0 -1$. Otherwise $w$ is a product of words of length $2$ and $l(w)-2$. If the latter has a factor of length 1, then $w$ has a subword of length $l(w) -2 -1 = l(w) - 2 \cdot 1 -1$. Otherwise, $w$ has a subword of length $l(w) - 2 \cdot 2$, to which the same argument can be applied. Since $l(w)$ is finite, by continuing this process we get  a word of length $l(w) - 2 t$ for some $t \ge 0$ such that it has a subword of length $1$. So,  $w$ has a subword of length $l(w) -2t -1$.
\end{proof}

\begin{remark}The definition above is formulated only in  $3$-mixing and $3$-sliding cases since by Propositions \ref{prop_repres} and \ref{prop_represZ} each word in these algebras is a sum of $3$-bounded words, so Lemma~\ref{Rem_step} is applicable.  
\end{remark}

\begin{definition}\label{def_step_def}
Let $\A$ be a  $3$-mixing or a $3$-sliding algebra and $\SS$ be its generating set. If there exists an irreducible 3-bounded word $w$ of length $l$ in $\SS$ such that 

1. $\sigma(w) = p$ and 

2. $\sigma(v) \ge p$ for each irreducible  3-bounded words $v$ of length~$l$,

then $w$ is called  a {\em $p$-step word}.

\end{definition}

The key idea behind these notions stems from the previously discussed concept of the gap $m_j - m_{j-1} = 2$ and can be formulated as follows.

\begin{lemma}\label{lem_ggap}
Let $\A$ be a finite dimensional  $3$-mixing $\F$-algebra, $\SS$ be its generating set, and $M=(m_1,\ldots,m_d)$ be the characteristic sequence of $\SS$. If there exists $j$ such that $m_{j} - m_{j-1} =2 $, then there exists an irreducible $3$-bounded $p$-step word of length $m_j$ with $p \ge 1$.
\end{lemma}

\begin{proof}
Consider an arbitrary $3$-bounded irreducible word $w$ of length $m_j$. 
Our aim is to demonstrate that $\sigma(w) \ge 1$.  

Firstly, as $m_{j-1} \ge 0$, we have $l(w) =m_j \ge 2$ and $\sigma(w)$ is well-defined, see Definition~\ref{def_step}. Additionally this implies  $w =w_1 w_2$, where $w_1$ and $w_2$ are irreducible words of positive lengths. There are three possibilities:

1. $l(w_1) < l(w_2)$. In this case $l(w_1) < 3$ as all of $l$-sprout sequences of the word $w$  start with $l(w_1)$ and $w$ is $3$-bounded. However, if $l(w_1)=1$, $l(w_2)$ would be equal to $m_j -1$. By Corollary \ref{cor_core} this would mean that $m_j -1$ belongs to $M$, which is impossible as $M$ is non-decreasing and $m_{j-1} < m_{j} -1$. Thus, $l(w_1) =2, l(w_2) = l(w)-2$ and $w$ does not contain subwords of length $m_{j} -1 = l(w)-1$, i.e. $\sigma(w) > 0$. 

2.  $l(w_2) >  l(w_1)$ is similar to the first case.

3. $l(w_1) = l(w_2)$. In this case $l(w_1)=l(w_2) < 3$ as all of $l$-sprout sequences of the word $w$  start with $l(w_1)=l(w_2)$ and $w$ is $3$-bounded. However, if $l(w_1)=l(w_2)=1$, then $m_j=l(w) = 2$, which means that $m_{j-1} = 0$. It is impossible for $M$ to not have elements equal to $1$, thus $l(w_1)=l(w_2) =2$, $m_j=4$ and $w$ does not contain subwords of length $3=m_{j} -1$, i.e. $\sigma(w) > 0$.

So, in each of the cases $\sigma(w) > 0$. Then $p\ge \sigma(w) \ge 1$, since $p$ is the minimal value of~$\sigma(w)$.

Since $m_j$ is in the characteristic sequence, there exists an irreducible word of length $m_j$ in $\SS$ by Corollary \ref{cor_core}. By Corollary \ref{cor_exbound} there exists at least one $3$-bounded irreducible word of length $m_j$. From the observations above it follows that the step function of all such words is at least $1$. Thus for the word $w_0$ with minimal possible value of the step function  we have $\sigma (w_0) \ge 1$. Finally, it follows from Definition \ref{def_step} that $w_0$ is a $\sigma(w_0)$-step word. 
\end{proof}

The following lemmas establish important properties of the step function~$\sigma$.

\begin{lemma}\label{lem_sparse} Let $\A$ be a finite dimensional  $3$-mixing $\F$-algebra, $\SS$ be its generating set and $w$ be a word of length $l$ in $\SS$.
If $\sigma(w)=p > 0$ and $l > 2p +3$ then:
\begin{enumerate}
\item There exists an $l$-sprout sequence $R(w)$ of $w$ such that the elements of $R(w)$ belong to the set $\{1,2\}$  and the first $(p+1)$ elements of this sequence are $(2,\ldots,2,1)$.
\item For every $j \in \{1,\ldots,p\}$ there exists a subword $w'$  of $w$ such that  $l(w')= l-2j+2$ and $\sigma(w')=p-j+1$.
\end{enumerate}
\end{lemma}
\begin{proof}
We will prove this statement using induction on $p$.

The base. For $p=1$ the only possible $j$ is $1$ and Item 2 is trivial. For Item 1 consider an $l$-sprout sequence $R$ of $w$ which satisfies Definition \ref{def_bound}. By our choice it contains only the entries 1 and 2.  Since $w$ does not contain subwords of length $l-1$, the first element of $R$ is 2, meaning that $w = uv$ where $l(u) =2 \le l(v)$ or $l(v) =2 \le u$. Assume that the first case holds. As $w$ has a subword of length $l-3 >2$, by Lemma \ref{lem_subduh} this subword is a subword of $v$ (as $u$ is too short), meaning that $v$ is a product of two words of length $l-3$ and $(l-2)-(l-3)=1$. In particular, the second element of $R$ is 1. The second case is similar.

The step. Assume that the statement holds for $p=1,\ldots, q$ and consider $p=q+1$, $q \ge 1$. Since $w$ does not contain subwords of length $l-1$ but it is 3-bounded, one of the factors in the last multiplication of $w$ is a subword of $w$ with length $l-2$. Denote this subword by $w_0$. As a subword of $w$, $w_0$ is also 3-bounded. We will demonstrate that $\sigma(w_0)=p-1$.

Firstly, observe that $w_0$ must contain a subword of $w$ with length $l-2p-1 = (l-2) - 2(p-1) -1$. This follows from the fact that $l-2p-1 > 2$, hence the other factor of the last multiplication in $w$, which has length 2, cannot contain it.  Secondly, $w_0$ cannot contain subwords of length $l-3,\ldots,l-2p+1$ as they would be subwords of $w$ as well. Thus, we can apply the induction hypothesis to $w_0$.

For Item 1, since $l(w_0) = l(w) -2 = l-2 \ge 2p+1 \ge 2$, by Lemma \ref{lem_seqexp} the sequence $(2, R(w_0))$ would be an $l$-sprout sequence of $w$ with the first $p+1$ elements equal to $(2,\ldots,2,1)$, which allows to set $R(w)=(2, R(w_0))$.

For Item 2, the case $j=1$ is trivial, the case $j=2$ is already established by $w_0$ and for greater $j$ we can find a subword $w'$ of $w_0$ (which in turn means that $w'$ is a subword of $w$) with the length $(l-2) - 2(j-1) +2 = l - 2j +2$ and $\sigma(w')=(p-1) - (j-1)+1=p-j+1$.
\end{proof}

\begin{lemma}\label{lem_pgrow}
Let $\A$ be a finite dimensional  $3$-mixing $\F$-algebra, $\SS$ be its generating set and $w$ be a word  of length $l$ in $\SS$, such that $w$ is equal to a product of words $v$ and $w_0$, satisfying the following properties:
\begin{enumerate}
\item $v$ has length 2.
\item $w_0$ is $3$-bounded.
\item $\sigma(w_0)=p-1$
\item $l>2p +3$.
\end{enumerate}

Then $w$ is $3$-bounded and $\sigma(w)=p$.
\end{lemma}
\begin{proof}
Consider the $l$-sprout sequence $R(w_0)$ obtained in Lemma \ref{lem_sparse}. By Lemma \ref{lem_seqexp} the sequence $(2, R(w_0))$ would be an $l$-sprout sequence of $w$, which means that $w$ is also $3$-bounded.  

Note that $w$ contains a subword of length $l-2p-1$ as its subword $w_0$ contains a subword of length $ (l-2) - 2(p-1) -1$. Also it does not contain any subwords of length $l-1$, and subwords of lengths $l-3,\ldots,l-2p+1$ would have been subwords of $w_0$ by Lemma \ref{lem_subduh} since $l-2p >2 = l(v)$. However, $w_0$ does not have subwords of such lengths, which means $\sigma(w)=p$.
\end{proof}

The final set of lemmas covers properties of $3$-bounded words specific to Malcev algebras.

\begin{lemma}\label{lem_ac}
Let $\A$ be a Malcev algebra, $\SS$ be its generating set and $w$ be an irreducible 3-bounded word in $\SS$.
If $\hat{w}$ is a word obtained from $w$ using transpositions of factors in the products, then $\hat{w}$ is also $3$-bounded and irreducible, while $\sigma(w)=\sigma(\hat{w})$.
\end{lemma}
\begin{proof}
As $\hat{w} = \pm w$ by anticommutativity and has the same length, the word $\hat{w}$ is also irreducible. Also note that the application of a transposition of factors does not affect the set of possible lengths of subwords or possible $l$-sprout sequences. Thus, $\hat{w}$ is also 3-bounded and $\sigma(w)=\sigma(\hat{w})$.
\end{proof}

\begin{lemma}\label{lem_smolswap2}
Let $\A$ be a Malcev algebra over a field $\F$ with $\char \F \neq 2$ and $\SS$ be its generating set. If $a,b,c \in \SS$, $w'$ is a $3$-bounded word of  length at least $2$ in $\SS$, then the following items are true: 
\begin{itemize}
\item $\sigma( (a b) (c w') )= 1$.

\item All of the summands on the right-hand side of Identity (\ref{eq_smolswap}) are $3$-bounded.

\item $\sigma ( (a c) (b w') ) = 1$.

\item $\sigma(z) = 0$ where $z$ is any summand other than  $(a c) (b w')$  on the right-hand side of Identity (\ref{eq_smolswap}).

\end{itemize}
\end{lemma}
\begin{proof}
It is straightforward to see that $\sigma ((a b) (c w') )  = 1$. Denote by $R$ a 3-bounded $l$-sprout sequence of $w'$.

We   check all the elements  on the right-hand side  of (\ref{eq_smolswap}) one by one.

1. $(ac)(b w')$ has the same set of lengths of subwords as $(a b) (c w')$ and the same possible $l$-sprout sequences, which means that it is $3$-bounded and the values of their step function are equal. Hence, they are equal to~1.

2. $a((w'b) c)$ is 3-bounded as it has an  $l$-sprout sequence $(1,1,1,R)$. Additionally,  $\sigma(a((w'b) c))=0$.

3. $b ((ca) w') $  is 3-bounded as it has an $l$-sprout sequence $(1,2,R)$. Additionally,  $\sigma(b ((ca) w') )=0$.

4. $c ((aw')  b)$ is 3-bounded as it has an $l$-sprout sequence $(1,1,1,R)$. Additionally,  $\sigma(c ((aw')  b))=0$.

5. $a ((w'c) b)$ is 3-bounded as it has an $l$-sprout sequence $(1,1,1,R)$. Additionally,  $\sigma(a ((w'c) b) )=0$.

6. $ c ((ba) w')$ is 3-bounded as it has an $l$-sprout sequence $(1,2,R)$. Additionally,  $\sigma(c ((ba) w')  )=0$.

7. $b ((aw') c)$ is 3-bounded as it has an $l$-sprout sequence $(1,1,1,R)$. Additionally, $\sigma(b ((aw') c))=0$.

\end{proof}

\begin{lemma}\label{lem_swap2}
Let $\A$ be a Malcev algebra over a field $\F$ with $\char \F \neq 2$ and $\SS$ be its generating set. If $a,b,c,d \in \SS$, and $w'$ is a 3-bounded word of  length at least $2$ in $\SS$ and $\sigma( (a b) ((cd) w') ) = p $ then the following items are true:
\begin{itemize}
\item All of the summands on the right-hand side of the identity (\ref{eq_swap}) are $3$-bounded.

\item $\sigma ( (c d) ((ab) w') ) =p$.

\item $\sigma(z) < p $ where $z$ is any summand other than  $(c d) ((ab) w')$  in the right-hand side of the identity~(\ref{eq_swap}).

\end{itemize}

\end{lemma}
\begin{proof}
It is easy to see that $p\ge 2$. Denote by $R$ a 3-bounded $l$-sprout sequence of $w'$.
We consider all the elements on the right-hand side of (\ref{eq_swap}) one by one and check that each statement holds true for each word.

1. $(c d) ((ab)  w')$ has the same set of lengths of subwords as $(a b) ((cd) w') $ and the same possible $l$-sprout sequences, which means that it is $3$-bounded and the values of step function are equal.

2. $( (cw') b) (d a)$   is 3-bounded as it has an $l$-sprout sequence $(2,1,1,R)$. Additionally,  $\sigma( ( (cw') b) (d a) )=1$.

3. $a ((bd) (cw'))$ is 3-bounded as it has an $l$-sprout sequence $(1,2,1,R)$. Additionally, $\sigma( a ((bd) (cw')) )=0$.

4. $ b((d(cw'))  a)$ is 3-bounded as it has an $l$-sprout sequence $(1,1,1,1,R)$. Additionally, $\sigma( b((d(cw'))  a) )=0$.

5. $d ((w' (ab))  c)$ is 3-bounded as it has an $l$-sprout sequence $(1,1,2,R)$. Additionally,  $\sigma(  d ((w' (ab))  c) )=0$.

6. $ (ab) ((dw')  c)$  is 3-bounded as it has an $l$-sprout sequence $(2,1,1,R)$. Additionally,  $\sigma( (ab) ((dw')  c) )=1$.

7. $c ((d b) (a w') ) $ is 3-bounded as it has an $l$-sprout sequence $(1,2,1,R)$. Additionally,  $\sigma( c ((d b) (a w') ) )=0$.

8. $c( d ((w' a) b)) $ is 3-bounded as it has an $l$-sprout sequence $(1,1,1,1,R)$. Additionally, $\sigma( c( d ((w' a) b))  )=0$.

9. $c ( b((ad ) w'))$   is 3-bounded as it has an $l$-sprout sequence $(1,1,2,R)$. Additionally,  $\sigma(c ( b((ad ) w'))  )=0$.

10. $c( a((dw') b))$ is 3-bounded as it has an $l$-sprout sequence $(1,1,1,1,R)$. Additionally,  $\sigma(c( a((dw') b)))=0$.

11.  $c ((d w')  (ab))$  3-bounded as it has an $l$-sprout sequence $(1,2,1,R)$. Additionally,  $\sigma(c ((d w')  (ab)))=0$.

So, each condition is justified, and the lemma is proved.
\end{proof}

\begin{lemma}\label{lem_bigswap}
Let $\A$ be a Malcev algebra  over a field $\F$ with $\char \F \neq 2$, $\SS$ be its generating set and $w$ be a $p$-step word of length $l$ in $\SS$ with $p \ge 2$ which has the form  $$w = (s_1 s_2) ( (s_3 s_4) ( \ldots  ( (s_{2p-1} s_{2p}) (s_{2p+1} w') ) \ldots )),$$ where $s_i \in \SS$ and $w'$ has the length $(l-2p-1)$. Then the word  $$(s_3 s_4) (\ldots  ( (s_{2p-1} s_{2p})  ((s_1 s_2)  (s_{2p+1} w') ) )\ldots )$$ is also $p$-step and equivalent to $w$.
\end{lemma}

\begin{proof}

We proceed by the induction on the index $t$ such that $$w_t=(s_3 s_4) (\ldots ( (s_{2t+1} s_{2t+2}) ( (s_{1} s_{2}) ( (s_{2t+3} s_{2t+4}) ( \ldots  ( (s_{2p-1} s_{2p}) (s_{2p+1} w') ) \ldots )$$ is a $p$-word and equivalent to $w$ for all $t=2,\ldots,p-1$.

The base. For $t=1$ we have $w_1 = (s_3 s_4) ( (s_1 s_2) ( \ldots  ( (s_{2p-1} s_{2p}) (s_{2p+1} w') ) \ldots )$. By Lemma \ref{lem_swap} $w = w_1 + x$, where $x$ is a linear combination of words with the value of step function being $0$ or $1$. As $w$ is a $p$-step this means that summands of $x$ are reducible and $x \in \L_{l-1}$. Thus, $w_1 \sim w$, and $w_1$ is irreducible by Lemma \ref{lem_sim}. Additionally, for the word  $v_1 =  (s_5 s_6) (\ldots  ((s_{2p-1} s_{2p}) (s_{2p+1} w') ) \ldots ) $ we have $\sigma(v_1) = p-2$ by Lemma \ref{lem_sparse}  Item 2, as it is the only subword of $w$ of the length $l-4$. 
Applying Lemma \ref{lem_pgrow}  to $  (s_1 s_2) v_1$ and $w_1 = (s_3 s_4) ( (s_1 s_2) v_1)$  we get that the word $w_1$ is $3$-bounded and $\sigma(w_1)=p$, which means $w_1$ is also a $p$-step word.

The step. Assume that the statement holds for $t=T$, with $T<p-1$. For $t=T+1$ consider the word $w_{T+1} =$ $$= (s_3 s_4) (\ldots ( (s_{2T+3} s_{2T+4}) ( (s_{1} s_{2}) ( (s_{2T+5} s_{2T+6}) ( \ldots  ( (s_{2p-1} s_{2p}) (s_{2p+1} w') ) \ldots ).$$

The word $w_{T+1} $ has a subword  $u_{T+1}=(s_{2T+3} s_{2T+4}) ( (s_{1} s_{2}) v_{T+1})$, where $$v_{T+1} =  (s_{2T+5} s_{2T+6}) ( \ldots  ( (s_{2p-1} s_{2p}) (s_{2p+1} w')  \ldots).$$

Note that by the induction hypothesis the word $$w_T = (s_3 s_4) (\ldots ( (s_{2T+1} s_{2T+2}) ( (s_{1} s_{2}) ( (s_{2T+3} s_{2T+4}) ( \ldots  ( (s_{2p-1} s_{2p}) (s_{2p+1} w') ) \ldots )$$ is a $p$-step and is  equivalent to $w$. The word $w_T$ has its own subword of length equal to $l(u_{T+1})$, that is $(s_{1} s_{2}) ( (s_{2T+3} s_{2T+4}) v_{T+1} ) =: u'_T$. By Lemma \ref{lem_sparse} Item 2, $\sigma(u'_T) = p- T \ge 2$ as it is the only subword of $w_T$ of length $l-2T$.

By Lemma \ref{lem_swap},
\begin{align}\label{eq_swapsh}
 u'_{T} = u_{T+1} +f_1 x_1 + \ldots + f_r x_r
\end{align}
 where $f_i = \pm 1$ and $x_i$ are the other summands from Identity \ref{eq_swap}. By Lemma \ref{lem_swap2}, each of $x_i$ is $3$-bounded and $\sigma(u'_{T}) = p - T > 1 \ge \sigma (x_i)$, while $\sigma(u_{T+1} ) = p -T$.  By sequentially multiplying the equality \eqref{eq_swapsh} by $(s_{2T+1} s_{2T+2}), \ldots, (s_3 s_4)$ on the left we would get $w_T = w_{T+1} +f_1 X_1 + \ldots + f_r X_r$, where all words are $3$-bounded and $\sigma(w_{T+1}) = p$. For each $i \in \{1, \ldots, r\}$ we can consider sequence $x_i^{(j)}$, defined inductively: $x_i^{(0)} = x_i$ and $x_i^{(j+1)} = (s_{2T+1-2j} s_{2T+2-2j}) x_i^{(j)}$, $j \le T-1$. It is straightforward to see that $X_i = x_i^{(T)}$.  By  Lemma \ref{lem_pgrow} $\sigma (x_i^{(j+1)}) = \sigma (x_i^{(j)})+1$, meaning that $\sigma (X_i) = \sigma( x_i^{(T)})  = \sigma( x_i^{(T-1)})+1 = \ldots = \sigma( x_i^{(0)}) +T =  \sigma(x_i) +T   < p-T+T = p$. Since $\sigma (X_i) < p$, these words cannot be irreducible and $w_{T+1} \sim w_T \sim w$.  Thus, $w_{T+1}$ is a $p$-step word and it is equivalent to~$w$.

This induction allows us to conclude that the word $$(s_3 s_4) ( \ldots  ( (s_{2p-1} s_{2p}) ((s_1 s_2) (s_{2p+1} w') ) ) \ldots )$$ corresponding to $t=p-1$ is a $p$-step word and is equivalent to~$w$.
\end{proof}

After all the preparations above we are ready to prove the key result of this section.

\begin{proposition}\label{prop_fin}
Let $\A$ be a Malcev algebra  over a field $\F$ with $\char \F \neq 2$, $\SS$ be its generating set and $w$ be a $p$-step word in $\SS$ of length $l$ with $p \ge 1$. Then there are at least two irreducible words of length $l-2$ linearly independent modulo~$\L_{l-3}$. 
\end{proposition}

\begin{proof}

Firstly since $w$ is a $p$-word, by Lemma \ref{lem_sparse} it has an $l$-sprout sequence starting with $p+1$ elements equal to $(2,\ldots,2,1)$. Thus, using the anticommutativity and Lemma \ref{lem_ac} we can construct the word 
\begin{equation} \label{w} \hat{w} = (s_1 s_2) ( (s_3 s_4) ( \ldots  ( (s_{2p-1} s_{2p}) (s_{2p+1} w') ) \ldots ),
\end{equation}  
where  $s_i \in \SS$ and $w'$ has length $l-2p-1$, such that  $\hat{w}=w$ in $\A$ and $\hat{w}$ has the same length and value of the step function as $w$. This means that now we can assume that $w$ is already in the form~\eqref{w}.

Consider firstly the case $p=1$.  This means $w = (s_1 s_2)(s_3 w')$. By Lemma \ref{lem_smolswap}  we get $w=(s_1 s_3) (s_2 w') +f_1 x_1 + \ldots + f_r x_r$, where  $f_i = \pm 1$ and $x_i$ are the other summands from the right-hand side of Identity \ref{eq_smolswap}. Note that by Lemma \ref{lem_smolswap2} we have $\sigma(x_i) = 0 < 1 =\sigma(w)$, which means that $x_i$  is reducible for each $i=1,\ldots, r$. Thus, the word $(s_1 s_3) (s_2 w')$ is  irreducible.

Note that the words $s_2 w'$ and $s_3 w'$  are irreducible as subwords of irreducible words. Assume that $s_2 w'$ and $s_3 w'$ are linearly dependent modulo $L_{l-3}$. Hence by Lemma \ref{lem_sim} the word $(s_1 s_2) (s_2 w')$ is irreducible. However $(s_1 s_2) (s_2 w') = s_1(w's_2 \cdot s_2) + w' (s_2 s_2 \cdot s_1) + s_2 (s_2 s_1 \cdot w') + s_2 (s_1 w' \cdot s_2)$, where the second word on the right-hand side is equal to zero as $s_2 s_2 = 0$ and other words are reducible as the value of the step function on them is again $0$. Thus the assumption is incorrect and $s_2 w'$ and $s_3 w'$ are linearly independent modulo~$\L_{l-3}$.

Now we consider the case $p\ge 2$. Consider the given word in the form~\eqref{w}: $$w  =  (s_1 s_2) ( (s_3 s_4) ( \ldots  ( (s_{2p-1} s_{2p}) (s_{2p+1} w') ) \ldots ).$$ 
By Lemma \ref{lem_bigswap} the word $w_p=(s_3 s_4) ( \ldots  ( (s_{2p-1} s_{2p}) ((s_1 s_2) (s_{2p+1} w') ) ) \ldots )$ is a $p$-step word equivalent to~$w$.
By Lemma~\ref{lem_smolswap}  
\begin{align}\label{eq_shsh}
(s_1 s_2) (s_{2p+1} w') = (s_1 s_{2p+1}) (s_{2} w') +f_1 y_1 + \ldots + f_r y_r
\end{align}
 where   $f_i = \pm 1$ and $y_i$ are the other summands from the right-hand side of Identity \ref{eq_swap}. Note that by Lemma \ref{lem_smolswap2} all of the words on the right hand side are $3$-bounded, and $\sigma((s_1 s_2) (s_{2p+1} w') ) = \sigma((s_1 s_{2p+1}) (s_{2} w') =1$, while $\sigma(y_i) =0$. By sequentially multiplying the equality \eqref{eq_shsh}  by  $(s_{2p-1} s_{2p}), \ldots, (s_3 s_4)$ on the left we would get $w_p = w'_p + f_1 Y_1 +\ldots + f_r Y_r$, where $w'_p = (s_3 s_4) (\ldots (  (s_1 s_{2p+1}) (s_{2} w') ) \ldots ),$ all words are $3$-bounded and $\sigma(w'_{p}) = p$. Moreover, by  Lemma \ref{lem_pgrow} applied $p-1$ times $\sigma(Y_i) = \sigma(y_i) + p-1 = p-1 < p$ . Thus all $Y_i$, $i=1,\ldots,r$,  are reducible, and $w'_p$ is irreducible.

Assume that  $(s_3 s_4) (\ldots (  (s_1 s_{2}) (s_{2p+1} w') ) \ldots )$ (the irreducible subword of $w$ of length $l-2$) and $(s_5 s_6) (\ldots (  (s_1 s_{2p+1}) (s_{2} w') ) \ldots )$ (the irreducible subword of $w'_p$ of length $l-2$) are linearly dependent modulo $L_{l-3}$. By Lemma \ref{lem_sim} it follows that $(s_1 s_2) ( (s_5 s_6) (\ldots (  (s_1 s_{2p+1}) (s_{2} w') ) \ldots ) ) $ is an irreducible word. However, as this word has the same set of lengths of subwords as $w$, it also has the same value of the step function. By applying Lemma \ref{lem_bigswap} to this word we would get that $(s_5 s_6) (\ldots (  (s_1 s_{2}) (s_{2} w') ) \ldots )$ is an irreducible word. Thus its subword   $(s_1 s_{2}) (s_{2} w')$ would be irreducible, which is incorrect due to the same reasoning as in the case $p=1$. Thus the assumption is false and the two words are not  linearly dependent modulo~$\L_{l-3}$.

\end{proof}

\begin{theorem}
Let $\A$ be a Malcev algebra  over a field $\F$ with $\char \F \neq 2$. Then it holds that $l(\A) \le d-1$ where $d = \dim \A$. If $\A$  is not Lie algebra, then it holds that $l(\A) \le d-2$ .
\end{theorem}
\begin{proof}
Consider a generating set $\SS$ of algebra $\A$ such that $l(\SS) = l(\A)$. If   $\dim \L_1(\SS) \le 2$, then by Proposition \ref{prop_liemalcev}  $\A$  is a Lie algebra.  In this case the result follows from  \cite[Proposition~4.7]{GutK21}.  Thus below we assume  $\dim \L_1(\SS) \ge 3$.

Consider the characteristic sequence $M=(m_1,\ldots,m_d)$ of~$\SS$. 

Note that \begin{align}\label{eq_vd} m_d = m_3 + (m_4 - m_3) + \ldots + (m_d - m_{d-1}). \end{align}  Since a Malcev algebra is $3$-mixing, by Lemma \ref{lem_stedstep}  we have $m_j - m_{j-1}  \le 2$. As the characteristic sequence is non-decreasing by Definition \ref{CharSeqB},  $m_j - m_{j-1}  \ge 0$. Thus, for $j \in \{4,\ldots,d\}$ we have $m_j - m_{j-1}  \in \{ 0,1,2 \}$.

Define the sets 
$I_0 = \{ j \vert  m_j - m_{j-1} =0 \}, \ I_1 = \{ j\vert   m_j - m_{j-1} =1 \}, \ I_2 = \{ j\vert   m_j - m_{j-1} =2 \}\subseteq \{4,\ldots,d\}$ and denote their cardinalities by $J_0= \vert I_0 \vert $, $J_1 = \vert I_1\vert $, $J_2 = \vert I_2 \vert$. By the observation above, $I_0 \sqcup I_1 \sqcup I_2 = \{4,\ldots,d\}$, thus $J_0 + J_1 + J_2 = d-3$. Additionally, by grouping respective terms of the equality (\ref{eq_vd}) we get $$m_d =m_3 + 0 J_0 + 1 J_1 + 2J_2,$$ and, since $m_3 = 1$, this means $$m_d =1 + 0 J_0 + 1 J_1 + 2J_2.$$

Assume $j \in I_2$, i.e. $m_j - m_{j-1} = 2$. In this case we have $j \ge 5$ as $m_3 = m_2 = m_1 = 1$, and either $m_4=1$ or $m_4 \ge 2$. By  Corollary \ref{cor_sum}  $m_4$ is equal to the sum of two previous elements, i.e. $m_4 =2$. By Lemma \ref{lem_ggap} $m_j - m_{j-1} = 2$ implies that there exists an irreducible $p$-step word of length $m_j$ in $\SS$ with $p \ge 1$. Thus by Proposition \ref{prop_fin} there are at least two non-equivalent words of length $m_{j-1}=m_j - 2$, which implies $m_{j-1} = m_{j-2}$ by Lemma \ref{lem_same}. Thus, for each $j \in I_2$ we have $j-1 \in I_0$ and the inequality $J_2 \le J_0$ is true, which allows us to conclude  by Lemma \ref{N=n}  that $$l(\A)  = l(\SS) = m_d = 1 + J_1 + 2J_2 \le 1+ J_0 + J_1 + J_2 = d-2.$$
\end{proof}

\section*{Data availability statement}

All data generated or analyzed during this study are included in this published article. 

\section*{Acknowledgment}

The work of the first author is partially financially supported by   the grant RSF 21-11-00283.

\end{document}